\documentclass[oneside,english]{article}

\usepackage{graphicx} 
\usepackage[T1]{fontenc}
\usepackage{amsmath}
\usepackage{amsthm}
\usepackage{amssymb}
\usepackage[all]{xy}
\usepackage{comment}

\usepackage{color}
\usepackage[nocompress]{cite}
\usepackage{tikz-cd}
\usepackage{mathtools}

\makeatletter
\newcommand{\nc}{\newcommand}
\theoremstyle{plain} 
\newtheorem{thm}{Theorem}
\nc{\bthm}{\begin{thm}} \nc{\ethm}{\end{thm}}
\newtheorem{prop}[thm]{Proposition}
\nc{\bprp}{\begin{prop}} \nc{\eprp}{\end{prop}}
\newtheorem{oques}[thm]{Open Question}
\newtheorem{fact}[thm]{Fact}
\nc{\bfct}{\begin{fact}} \nc{\efct}{\end{fact}}
\newtheorem{prob}[thm]{Problem}
\nc{\bprb}{\begin{prob}} \nc{\eprb}{\end{prob}}
\newtheorem{lem}[thm]{Lemma}
\nc{\blem}{\begin{lem}} \nc{\elem}{\end{lem}}
\newtheorem{claim}[thm]{Claim}
\nc{\bclm}{\begin{claim}} \nc{\eclm}{\end{claim}}
\newtheorem{cor}[thm]{Corollary}
\nc{\bcor}{\begin{cor}} \nc{\ecor}{\end{cor}}
\newtheorem{conj}[thm]{Conjecture}
\nc{\bcnj}{\begin{conj}} \nc{\ecnj}{\end{conj}}
\theoremstyle{definition}
\newtheorem{defn}[thm]{Definition}
\nc{\bdfn}{\begin{defn}} \nc{\edfn}{\end{defn}}
\newtheorem{observation}[thm]{Observation}
\nc{\bobs}{\begin{observation}} \nc{\eobs}{\end{observation}}
\theoremstyle{remark}
\newtheorem{rem}[thm]{Remark}
\nc{\brem}{\begin{rem}} \nc{\erem}{\end{rem}}
\newtheorem{cnv}[thm]{Convention}
\nc{\bcnv}{\begin{cnv}} \nc{\ecnv}{\end{cnv}}
\newtheorem{exam}[thm]{Example}
\nc{\bexm}{\begin{exam}} \nc{\eexm}{\end{exam}}

\nc{\bpf}{\begin{proof}} \nc{\epf}{\end{proof}}
\nc{\be}{\begin{enumerate}}
	\nc{\ee}{\end{enumerate}}
\nc{\bi}{\begin{itemize}}
	\nc{\itm}{\item}
	\nc{\ei}{\end{itemize}}
\nc{\invlim}{\lim_{\leftarrow}}
\nc{\dirlim}{\lim_{\rightarrow}}
\nc{\mm}{\mathbf{m}}
\nc{\nn}{\mathbf{n}}
\nc{\kk}{\mathbf{k}}
\nc{\FF}{\mathcal{F}}
\nc{\CC}{\mathcal{C}}
\nc{\Span}{\operatorname{span}}
\nc{\Img}{\operatorname{Im}}
\nc{\rank}{\operatorname{rank}}
\nc{\proj}{\operatorname{proj}}
\nc{\F}{\mathbb{F}}
\nc{\Z}{\mathbb{Z}}
\nc{\Q}{\mathbb{Q}}
\nc{\Br}{\operatorname{Br}}

\title{ Groups of profinite type and profinite rigidity}

\author{Tamar Bar-On and Nikolay Nikolov}
\date{\today}

\begin{document}
	
	\maketitle
	
	\begin{abstract}
		We say that a group $G$ is of \textit{profinite type} if it can be realized as a Galois group of some field extension. Using Krull's theory, this is equivalent to the ability of $G$ to be equipped with a profinite topology. We also say that a group of profinite type is \textit{profinitely rigid} if it admits a unique profinite topology. In this paper we study when abelian groups and some group extensions are of profinite type or profinitely rigid. We also discuss the connection between the properties of profinite type and profinite rigidity to the injectivity and surjectivity of the cohomology comparison maps, which were studied by Sury and other authors.
	\end{abstract}
	\section{Introduction}
	Let $G$ be a group. An old result states that if $G$ is finite then it can be realized as a Galois group of some field extension. It is still an open question, however, if every finite group can be realized as a Galois group over $\mathbb{Q}$- this is "the inverse Galois problem".
	
	For infinite groups the situation is much more mysterious. In his groundbreaking paper from 1928, \cite{krull1928galoissche} Wolfgang Krull proved that every Galois group can be endowed with a group topology which is called "the Krull topology", and is compact, Hausdorff and totally-disconnected. This work has been completed in 1964 by Douady, \cite{douady1964determination}, who proved that every topological group which is compact, Hausdorff and totally disconnected can be realized as a Galois group over some field. Such groups are known as \emph{profinite groups} and are exactly the groups which are inverse limit of finite groups. Hence, they admit a local basis of the identity that consists of some of their finite-index subgroups. We call the Krull topology of a profinite group a \textit{profinite topology}. This should not be confused with the (generally noncompact) topology which has  all subgroups of finite index as a local basis of the identity. For clarity we will refer to the latter as the "finite-index topology". These two topologies might differ on a given profinite group (see \cite{ribes2000profinite}). 
	
	Let $G$ be an abstract group. We say that $G$ is of profinite type if it can be given a profinite group topology. Hence, these are precisely the groups that can be realized as the Galois group of some field extension. 
	There are a few well-known restrictions on the algebraic structure of a profinite group, and hence of abstract groups of profinite type. For example, since the intersection of all open subgroups of a profinite group is trivial, a group of profinite type is residually finite. In addition, following Baire's Category Theorem, a group of profinite type is either finite or uncountable.   
	
	Identifying all groups of profinite type among the infinite abstract groups is an extremely difficult task. However, there are some natural candidates we shall examine, namely the residually finite extensions of profinite groups and in particular subgroups of finite index of profinite groups.
	
	Given a group of profinite type it is natural to ask how many different profinite topologies can be given on it, up to a continuous isomorphism. We say that a group $G$ of profinite type is \textit{profinitely rigid} (or just rigid) if $G$ admits a unique profinite topology. A first example is the variety of finite groups with the discrete topology. We say that $G$ is \emph{weakly rigid} if all profinite topologies on $G$ are equivalent as topological groups. In that case we say that $G$ admits a \textit{unique profinite type}. For example if we assume the Continuum Hypothesis the profinite group $(\Z/p\Z)^\mathbb N$ is  weakly rigid but not rigid.
	
	The object of this paper is to investigate the groups of profinite type, profinite rigidity and the connections between different possible profinite topologies on a given abstract group.
	
	The paper is organized as follows: In  Section 2 we give a necessary and sufficient criterion for some abelian groups to be of profinite type with applications. In Section 3 we prove some general results on profinite type groups and  profinite rigidity. In addition, we study profinite invariants of groups of profinite type - i.e. topological properties that are preserved under abstract isomorphism. In section 4 we discuss connections with cohomological goodness of abstract and profinite groups.
	
	\section{Abelian groups of profinite type}
	We deal first with torsion abelian groups. By \cite[Lemma 4.3.7]{ribes2000profinite}, a torsion abelian profinite group must be of finite exponent. Hence it is enough to deal with abelian groups of finite exponent. 
	\begin{prop}\label{torsion abelian groups}
		Let $G$ be an abelian group of exponent $n$. Let $n=\prod_{i} p_i^{t_i}$ be a factorization of $n$ as a product of prime powers. By \cite[Theorem 6]{kaplansky2018infinite}  
		$G$ is isomorphic to \[ \bigoplus_{i}\left (\bigoplus_{j_i=1}^{t_i}(\bigoplus_{\mm_{j_i}} C_{p_i^{j_i}}) \right )\]
		for some cardinals $\textbf m_{j_i}$.
		
		$G$ is of profinite type if and only if for every $j_i$, either $\mm_{j_i}$ is finite, or there exists some cardinal $\mathbf{n}_{j_i}$ such that $2^{\nn_{j_i}}=\mm_{j_i}$.
	\end{prop}
	\begin{rem}
		Assuming the Generalised Continuum Hypothesis, Proposition \ref{torsion abelian groups} can be rephrased as: For every $i_j$, $\mm_{i_j}$ is either zero or a successor cardinal.
	\end{rem}
	\begin{proof}
		$\Leftarrow$ First notice that if $\{A_i\}_{i\in I}$ are groups of profinite type then so is $\prod_I A_i$. So, it is sufficient  to show  that if $\mm_{j_i}$ is either finite or equal to $2^{\nn_{j_i}}$ then $\bigoplus_{\mm_{j_i}} C_{p^{j_i}}$ is of profinite type.  Indeed, for $\mm_{j_i}$ finite the claim is trivial. Otherwise, consider $\prod_{\nn_{j_i}} C_{p^i}$, which is a profinite group. One easily checks that it is a free module over $\mathbb{Z}/p^{\nn_{j_i}}$ and thus is isomorphic to $\bigoplus_I C_{p^{\nn_{j_i}}}$. For cardinality reasons, $|I|=2^{\nn_{j_i}}=\mm_{j_i}$.
		
		$\Rightarrow$ For an integer $m \in \mathbb N$ we denote by $G[m]$ the subgroup of elements of $G$ of order dividing $m$. If $G$ is an abelian profinite group then $G[m]$ is a closed subgroup of $G$.
		Now assume that $G$ has finite exponent $n=\prod_i {p_i}^{t_i}$. Each group $H_i:=G[p_i^{t_i}]= \bigoplus_{j_i=1}^{t_i}(\bigoplus_{\mm_{j_i}} C_{p_i^{j_i}})$ also has profinite type.  Let $1\leq k\leq t_i$ and let $L=p^{k-1} H_i= \bigoplus_{j=k}^{t_i} (\bigoplus_{\mm_{j}} C_{p_i^{j-k+1}})$ which also has profinite type. Therefore $L[p]/ (L[p] \cap pL) \cong \bigoplus_{\mm_k} C_{p_i}$ is of profinite type. The only abelian profinite groups of exponent $p_i$ are the direct products of copies of $C_{p_i}$ (see \cite[Theorem 4.3.8]{ribes2000profinite}), so again for cardinality reasons, $\mm_k$ must be either finite or $2^{\nn_k}$ for some cardinal $\nn_k$. 
	\end{proof}
	Now we can give a criteria for a torsion free abelian group to be of profinite type, considering the torsion case is already known. First we need few lemmas.
	\begin{lem}\label{free and lifting basis}
		Let $G$ be a torsion free abelian group. Then for every prime number $p$ and natural number $n$, $G/p^nG$ is free module over $\mathbb{Z}/p^n\mathbb{Z}$. Moreover, if $\{g+p^nG\}_{g\in I}$ is a basis of $G/p^nG$ as a free module over $\mathbb{Z}/p^n\mathbb{Z}$ then $\{g+p^{n+1}G\}_{g\in I}$ is a basis of $G/p^{n+1}G$ as a free module over $\mathbb{Z}/p^{n+1}\mathbb{Z}$.
	\end{lem}
	\begin{proof}
		We prove the first claim by induction on $n$. For $n=1$ $G/pG$ is vector space over $\mathbb{F}_p$ so the claim is trivial. Now assume that $G/p^nG$ is a free module over $\mathbb{Z}/p^n\mathbb{Z}$, and for every $x\in G/p^{n+1}G$ write $\bar{x}$ for the image of $x$ in the quotient $G/p^nG$. We want to show that if $G/p^nG\cong \bigoplus_I\langle \bar{x_i}\rangle$ then $G/p^{n+1}G\cong \bigoplus_I\langle x_i\rangle$.
		
		First we show that if $y\in G/p^{n+1}G$ satisfies $p^ky=0$ for some $k<p^{n+1}$ then $y$ has a $p^{n+1-k}$'th root. Indeed, express $y=g+p^{n+1}G$ for some $g\in G$. Then $p^kg\in p^{n+1}G$ meaning that there exists $g'\in G$ such that $p^kg=p^{n+1}g'$. In particular $p^k(g-p^{n+1-k}g')=0$. However, $G$ is torsion free, hence $g=p^{n+1-k}g'$ and in particular $g+p^{n+1}G=p^{n+1-k}(g'+p^{n+1}G)$. 
		
		Now, assume that $G/p^nG\cong \bigoplus_I\langle \bar{x_i}\rangle$. Since $pG/p^{n+1}G$ is the Frattini subgroup of $G/p^{n+1}G$ we have that $\{x_i\}_I$ is a generating set of $G/p^{n+1}G$.
		
		It remains to show that $\{x_i\}_I$ is an independent set over $\mathbb{Z}/p^{n+1}\mathbb{Z}$. Let $J\subseteq I$ be a finite subset and assume that $\sum_J\alpha_ix_i=0$ for some integers $\alpha_i\ne0\mod p^{n+1}$ for all $i\in J$. Since $\{\bar{x_i}\}_J$ are independent over $\mathbb{Z}/p^n\mathbb{Z}_p$ then $p^n| \alpha_i$ for all $i\in J$. Put $\alpha_i=p^n\beta_i$ for $(\beta_i,p)=1$. We conclude that $\sum_J\beta_ix_i$ is annihilated by $p^n$, but that means that $\sum_J\beta_ix_i\in pG/p^{n+1}G$. Since the set $\{x_i\}$ is spanning, there is some finite subset $J'\subseteq I$ such that $\sum_J\beta_ix_i=p\sum_{J'}\gamma_ix_i$. Dividing by $p^nG$ again, for every $i\in J\cap J'$ $\beta_i+p\gamma_i$ must be divided by $p^n$ and for every $i\in J\setminus J'$ $\beta_i$ must be divided by $p^n$, a contradiction. 
	\end{proof}
	\begin{prop}\label{torsion free abelian groups of profinite type}
		Let $G$ be an abstract torsion free abelian group. Then $G$ is of profinite type if and only if the natural maps $G\to G/nG$ induce an isomorphism $G\to \invlim G/nG$ where $n$ runs over the naturals with the division relation, and for every $n$, $G/nG$ is of profinite type.
	\end{prop}
	\begin{proof}
		The first direction is trivial.
		
		Now assume that every $G/nG$ admits a profinite topology, we shall construct a profinite topology on $\invlim G/nG$. By \cite[Theorem 1]{kaplansky2018infinite} for every $n$, $G/nG\cong \bigoplus_I G/{p^{k_{p,n}}}G$ where $I$ is the set of all prime numbers $p$ that divide $n$ and for all $p\in I$, $k_{p,n}$ is the maximal power of $p$ such that $p^{k_{p,n}}$ divides $n$. By Lemma \ref{free and lifting basis}, for every prime $p$ there exists some cardinal $\mm_p$ such that $G/p^{k_{p,n}}G\cong \bigoplus_{\mm_p}C_{p^{k_{p,n}}}$, for every $n$. By Proposition \ref{torsion abelian groups}, since $G/nG$ is of profinite type, for every $p$ $\mm_p$ is either finite of the form $2^{\nn_p}$ for some cardinal number $\nn_p$. It is enough to construct profinite topologies on $G/p^kG$ for every $p$ and $k$ such that the natural epimorphisms $G/p^{k+1}G\to G/{p^k}G$ are continuous. Then taking the product topologies will lead to compatible topologies on $\{G/nG\}_{n\in \mathbb{N}}$. Recall that constructing a topology on $G/p^kG$ is equivalent to choosing an isomorphism $G/p^kG\to \prod_{\kk_p} C_{p^k}$, where $\kk_{p}=\mm_{p}$ in case $\mm_p$ is finite, and $\kk_p=\nn_p$ otherwise. Since the maps $\prod_{\kk_p} C_{p^{k+1}}\to \prod_{\kk_p} C_{p^k}$ that are defined by $x\to x\mod p^k$ are continuous with respect to the product topology for every natural number $k$, it is enough to choose isomorphisms $\varphi_k:G/p^kG\to \prod_{\kk_p} C_{p^k} $ which makes the following diagram commutative:
		\[
		\xymatrix@R=14pt{ G/p^{k+1}G\ \ar[r]^{\varphi_{k+1}}\ \ar[d]_{x\to x\bmod p^k}& \prod_{\kk_p} C_{p^{k+1}}\ \ar[d]^{x\to x\mod p^k}\\
			G/p^kG \ar[r]_{\varphi_k} & \prod_{\kk_p} C_{p^k}\\
		}
		\]
		We build such maps by recursion. Assume $\varphi_k$ is already defined and choose $\{x_i\}_{\mm_p}$ to be a basis of $G/p^kG$ as a free module over $\mathbb{Z}/p^k\mathbb{Z}$. Then $\{\varphi(x_i)\}_{\mm_p}$ is a basis for $\prod_{\kk_p}C_{p^k} $. By Lemma \ref{free and lifting basis} these bases can be lifted to bases $\{y_i\}_{\mm_p}$ and $\{z_i\}_{\mm_p}$ of $G/p^{k+1}G$ and $\prod_{\kk_p}C_{p^{k+1}} $ correspondingly. Define an isomorphism $G/p^{k+1}G\to \prod_{\kk_p}C_{p^{k+1}}$ by sending $y_i$ to $z_i$ for every $i\in \mm_p$.
		
	\end{proof}
	Proposition \ref{torsion free abelian groups of profinite type} cannot be extended to general abelian groups, as can be shown in the following example:
	\begin{exam}
		Let $p$ be a prime. 
		For a positive integer $n$ and a set $Y$ we define \[ L_n(Y):= \bigoplus_{y \in Y} \frac{\Z}{p^n \Z} y \] to be the free module over $\Z /p^n \Z$ with basis $Y$.
		
		Let $\{X_i\}_{i=1} ^ \infty$ be a sequence of disjoint sets $X_i$ each of cardinality $2^{\aleph_0}$.
		
		For $n \in \mathbb N$ define 
		\[ G_n= \left( \oplus_{i=1}^{n-1} L_i(X_i) \right ) \oplus L_n (\cup_{j \geq n} X_j)  \]
		
		In particular $G_1= L_1( \cup_{i=1}^\infty X_i)$ is $\mathbb F_p$ vector space with basis $\cup_{i=1}^\infty X_i$.
		Note that $G_n= G_{n+1}/p^nG_{n+1}$ and let $f_{n,n+1}: G_{n+1} \rightarrow G_n$ be the quotient map $g \mapsto g+p^nG_{n+1}$ ($g \in G_{n+1}$).
		
		Let $G= \underleftarrow{\lim} \ G_n$ be the inverse limit of the groups $G_n$ with the homomorphisms $f_{n,n+1}$ above.
		Note that each $G_n$ is a group of profinite type, in fact $G_n$ is isomorphic to the profinite group $\oplus_{i=1}^n (\Z/p^i\Z)^{\aleph_0}$. We prove below that $G$ is not of profinite type.
		
		First we note that $G/p^n G \simeq G_n$. In particular we may identity $G/pG$ with $G_1=L_1(\cup_{i=1}^\infty X_i)$.
		
		For each $k$ let $h_k : G/pG \rightarrow p^kG/p^{k+1}G$ be defined by $h(g +pG)= p^kg +p^{k+1}G$ for all $g \in G$.
		Let $W_k = \ker h_k \subset G/pG=G_1$. By computing $h_k$ in $G/p^{k+1}G \simeq G_{k+1}$ we have $W_k= \sum_{i=1}^k \mathbb F_p X_i \subset G_1$.
		
		Suppose that $G$ is of profinite type. Then each subgroup $W_k< G_1$ is closed with respect to the induced profinite topology on $G_1=G/pG$. In particular since $|G_1:W_k|=\infty$ the group $W_k$ is nowhere dense in $G_1$. However $G_1= \cup_{k=1}^\infty W_k$ which contradicts the Baire Category theorem. Therefore $G$ is not of profinite type.
		
	\end{exam}
	\begin{cor}
		The class of abstract (abelian) groups of profinite type is not closed under inverse limits.
	\end{cor}
	\begin{prop}
		Let $G$ be a profinite abelian group which is either torsion ot torsion free, and let $U\leq G$ be a subgroup of finite index. Then $U$ is of profinite type.
	\end{prop}
	\begin{proof}
		First we show that it is enough to prove the claim for pro-$p$ groups. Indeed, recall that every profinite abelian group is isomorphic to the the direct product of its $p$-Sylow subgroups (see, for example, \cite[Proposition 2.3.8]{ribes2000profinite}). Therefore $G\cong \prod_pG_p$ where $G_p$ is a pro-$p$ abelian group, and $p$ runs over the set of all prime numbers. Denote $n=[G:U]$. Then $nG\leq U$.  For every $p$ prime to $n$, $nG_p=G_p$. Denote $A=\prod_{p|n}G_p, B=\prod_{(p,n)=1}G_p$. Then $\{0\}\times B\leq U$. Hence, $U\cong (U\cap A)\times B$. The group $B$ is already profinite, so it is enough to prove that $U\cap A$ is of profinite type. Replacing $G$ by $A$ and $U$ by $U\cap A$ we may assume that $G=\bigoplus_{i=1}^kG_{p_i}$. Thus $nG_p=\bigoplus_{i=1}^knG_{p_i}\leq U$. Since $G/nG$ is a torsion abelian group whose $p$-torsion parts are precisely $G_p/nG_p$, $U/nG\cong \bigoplus_{i=1}^k(U\cap G_{p_i})nG/nG$. However, $nG\cong \bigoplus_{i=1}^k(nG\cap G_{p_i})$ while for every $i$, $nG\cap G_{p_i}=nG_{p_i}\leq U$, so we are done.
		
		From now on we assume that $G$ is a pro-$p$ group. It is enough to consider the case $[G:U]=p$.
		
		First case: Assume that $G$ is a torsion pro-$p$ group and $[G:U]=p$. Let $x\in G\setminus U$ be of minimal order. Denote $y=px\in U$. We claim that $\langle y \rangle$ is a direct summand of $U$. By \cite[Theorem 5]{kaplansky2018infinite}, (since $U/\langle y\rangle $ is a direct sum of cyclic groups, being an abelian group of finite exponent), it is enough to prove that $\langle y \rangle$ is a \textit{pure} subgroup. By \cite[Lemma 7]{kaplansky2018infinite} this is satisfied in case the elements of order $p$ in $\langle y\rangle$ have roots of the same maximal $p$-th power in $U$ as in $\langle y\rangle$. Let $o(x)=p^k$. The elements of order $p$ in $\langle y \rangle$ are precisely $p^{r-1}mx$ for integers $m$ coprime to $p$. The highest $p$-th power root of these elements in $p^{r-2}$. Assume $p^{r-1}mx=p^{r-1}u$ for some $u\in U$, then $mx-u\in G\setminus U$ has order strictly less then the order of $x$, a contradiction.
		
		Express $U=\langle px\rangle  \times U'$ for some $U'$, we get that $G= \langle x\rangle \times U'$. Being finite, $\langle x\rangle$ is a closed subgroup in $G$, hence the quotient topology on $U'$ from its isomorphism with $G/ \langle x \rangle$ is profinite. Finally, we define a profinite topology on $U$ to be the product topology of the profinite topology on $U'$ and the discrete topology on the finite subgroup $\langle px\rangle$. 
		
		Second case: Assume that $G$ is torsion free. Consider $G/pG$. Since $pG\leq U$, $[G/pG:U/pG]=p$. By the torsion case, we can define a profinite topology on $G/pG$ for which $U/pG$ is closed. By the proof of Proposition \ref{torsion free abelian groups of profinite type}, this topology can be lifted to a profinite topology $\tau$ on $G$. Observe that $U$ is closed in $\tau$ being the inverse image of $U/pG$. 
	\end{proof}
	In Corollary \ref{abelain extension of group of profinite type} we will see that the converse holds in general: If $G$ is an abelian group which possesses a finite-index subgroup of profinite type, then $G$ is of profinite type. However in general an extension of a finite abelian group by a profinite abelian group might not be of profinite type, as the following example shows:
	\begin{exam}
		Let $W$ be the two dimensional vector space over $\F_p$ with basis $e_1$ and $e_2$ and define an automorphism $f$ to act on $W$ by $f(e_1)=e_1, f(e_2)=e_1+e_2$.
		Note that $f$ has order $p$ and the centralizer of $f$ in $W$ is the span of $e_1$.
		
		Now let $H_0$ be the direct sum $W_1 \oplus W_2$ where $W_1$ is a vector space over $\F_p$ of dimension $2^{\mathbb{N}}$  on which $f$ acts as the identity, and $W_2$ is a direct sum of countably many copies of $W$ with the action of $f$ on $W$ given above. 
		Since $H_0$ has dimension $2^{\mathbb{N}}$ as a vector space over $\F_p$, we have by Proposition \ref{torsion abelian groups} that $H_0$ is of profinite type.
		We claim that the semidirect product $G$ of $H_0$ with the cyclic group $\langle f\rangle=C_p$ with the action by conjugation described above is not of profinite type. Suppose it is. The centralizer $C_G(f)$ of $f$ in $G$ is generated by $f$ and the centralizer $C_{H_0}(f)$ of $f$ in $H_0$ which is a subspace of countably infinite codimension in $H_0$. Hence $C_G(f)$ is a closed subgroup of countably infinite index in $G$. But the quotient $G/G_G(f)$ is then a countably infinite profinite group, a contradiction.
		
	\end{exam}
	\section{Profinite rigidity}
	The easiest example of a profinite group $G$ which is not profinitely rigid is $G=C_p^\mm$  with the product topology $\tau$ for an infinite cardinal $\mm$. Indeed, given any non-open subgroup $L$ of index $p$ of $G$ there is an abstract automorphism $f \in \mathrm{Aut}(G)$ such that $f(L)$ is an open subgroup of $G$
	and so $G$ admits a different profinite topology $f^{-1}(\tau)$. 
	However, assuming the Generalised Continuum Hypothesis the cardinal $\mm$ is determined from $2^\mm$ and thus all profinite topologies on $G$ are isomorphic as profinite groups, i.e. $G$ is weakly rigid.
	
	\begin{prop} The map $\mm \mapsto 2^\mm$ is injective on cardinals if and only if all torsion abelian profinite groups are weakly rigid.
	\end{prop}   
	\begin{proof}
		If $\mm_1$ and $\mm_2$ are two different cardinals such that $2^{\mm_1}=2^{\mm_2}$ then the groups $G_1=(C_2)^{\mm_1}$ and $G_2=(C_2)^{\mm_2}$ are abstractly isomorphic being vector spaces over $\mathbb F_2$ with the same cardinality. On the other hand $G_i$ has $\mm_i$ open subgroups and $\mm_1 \not =\mm_2$, hence $G_1$ is not weakly rigid.
		
		Suppose now that the cardinal $\mm$ is uniquely determined by $2^\mm$. Let $n=p_1^{t_1} \cdots p_k^{t_k}$ be a positive integer with factorisation into prime powers $p_i^{t_i}$. Any torsion abelian profinite group $G$ of exponent $n$ is isomorphic to \[ \prod_{i=1}^k \left (\prod_{j=1}^{t_i} (C_{p_i^j})^{\mm_{i,j}} \right ) \]
		for some cardinals $\mm_{i,j}$. In the course of the proof of Proposition \ref{torsion abelian groups} we showed that the cardinals $2^{\mm_{i,j}}$ are uniquely determined from the structure of $G$ as an abstract group. Therefore the cardinals $\mm_{i,j}$ are uniquely determined and hence $G$ is weakly rigid. 
		
	\end{proof}
	
	We next give a necessary and sufficient condition for a profinite topology of a finite index normal subgroup to induce a profinite topology on the whole group.
	\begin{lem}\label{going up}
		Let $G$ be an abstract group and $U\unlhd_f G$ a finite index normal subgroup such that $U$ is of profinite type. Then the topology on $U$ can be extended to a profinite topology on $G$ such that $U\unlhd_o G$ iff for all $g\in G$ and $H\leq _oU$, $gHg^{-1}\leq_o U$.
	\end{lem}
	\begin{proof}
		$\Rightarrow$ If $G$ is a profinite group and $U\unlhd_o G$ then every open subgroup of $U$ is open in $G$. By a basic exercise in topology, a compact Hausdorff topology cannot be strictly included in any other compact  topology. Thus, the topology that $G$ induces on $U$ must be equal to the original topology in $U$. Hence, by continuity of the multiplication, $gHg^{-1}\leq _oU$  for all $g\in G$ and $H\leq _oU$.
		
		$\Leftarrow$ Let $G$ and $U$ be as in the proposition. We define a topology on $G$ by letting the open subgroups of $U$ serve as a set of neighborhoods of the identity. As the basic open groups are cosets of subgroups, the inverse operation is continuous. We shall show that the product operation is continuous. Let $H\leq U$ be an open subgroup and $g_1,g_2\in G$, we shall find open subgroups $H_1,H_2$ such that $g_1H_1g_2H_2\leq g_1g_2H$. Take $H_1=g_2Hg_2^{-1}$ and $H_2=H$ will do. Observe that by assumption $g_2Hg_2^{-1}$ is indeed open. We get that $G$ is a topological group. As a finite union of profinite spaces, $G$ is profinite.  
	\end{proof}
	\begin{cor}\label{abelain extension of group of profinite type}
		Any abelian group possessing a finite-index subgroup of profinite type is of profinite type itself. 
	\end{cor}
	Recall that a profinite group $G$ is called \textit{profinitely rigid} if it admits a unique profinite topology. In particular, every abstract automorphism of $G$ is in fact continuous. A basic class of examples is the class of finite groups with the discrete topology. Another elementary class of examples is the class of \textit{strongly complete} profinite groups, and in particular, finitely generated profinite groups (see \cite{SegalAndNikolov} for the full proof). Recall that a profinite group is called \textit{strongly complete} if every subgroup of finite index is open. As the finite-index topology on it is compact-Hausdorff, any properly weaker topology can not be Hausdorff. Another class of examples was presented by Kiehlmann in his paper \cite{kiehlmann2013classifications}, by the following theorem:
	\begin{thm}
		Let $G\cong \prod _I A_i$ be a direct product of finite groups all having trivial center. Then $G$ is profinitely rigid.
	\end{thm} In particular, Kiehlmann's Theorem can be applied to all semi-simple profinite groups. The above discussion is the core of the following surprising result:
	\begin{prop} \label{ex1}
		There exists a nonstrongly complete profinite group $S$ whose finite-index normal subgroups are either open, and thus profinite, or are not of profinite type. In fact we can take $G=S^I$, where $S$ is a nonabelian finite simple group and $I$ is an infinite set.
	\end{prop}
	
	The proof relies on the following.
	\begin{lem} \label{lss} Let $S$ be a nonabelian finite simple group and let $G=S^I$ for a set $I$. Then any finite quotient of $G$ is semisimple.
		
	\end{lem}
	
	\begin{proof}
		Let $D$ be a finite homomorphic image of $G$ and observe that $D$ satisfies all group laws of $S$. Let $F=F_n$ be a free group of rank $n \in \mathbb N$ which has $D$ as a homomorphic image. Let $N= \cap_{h \in Z} \ker h$ where $Z$ is the set of all homomorphisms $h: F \rightarrow S$. The set $Z$ is finite  since $F$ is a finitely generated group and thus $Q=F/N$ is a finite subdirect product of $\prod_{h \in Z} T_h$ where each $T_h=h(Q) \leq S$. By construction $N$ is subgroup of all $n$-variable laws of $S$ and since $D$ satisfies these laws it follows that $D$ is a homomorphic image of $Q$, say $D=f(Q)$ for a surjective homomorphism $f: Q \rightarrow D$. 
		
		Suppose that $T_h$ is a proper subgroup of $S$ for some $h \in Z$ and let $Q'=Q \cap \ker h$. We have that $|Q:Q'|=|T_h|<|S|$. Hence $f(Q')$ is a normal subgroup of $f(Q)=D$ of index less than $|S|$. But $D$ cannot have a proper quotient of size less than $S$ since the only finite simple image of $G$ is $S$. In conclusion $f(Q')=D$ and in this way we can remove all factors $T_h$ which are proper subgroups of $S$ and deduce that $D=f(Q_0)$ where the subgroup $Q_0 \leq Q$ is a subdirect product of finitely many copies of $S$. But then both $Q_0$ and its quotient $D$ are semisimple by a standard argument.
	\end{proof}
	
	We can now prove Proposition \ref{ex1}
	\begin{proof} 
		Denote by $\tau$ the product topology on $G$, thus $(G,\tau)$ is a profinite group. By \cite[Example 4.2.12]{ribes2000profinite} $G$ is not strongly complete. Let $U$ be a normal subgroup of finite index in $G$. Assume that $U$ is of profinite type, we will prove that $U$ must be open in $G$. Every finite image of $U$ is a finite image of $G$ and hence semisimple by Lemma \ref{lss}. Being the inverse limit of finite semisimple groups, $U$ is therefore a semisimple profinite group. Hence by Kiehlmann's Theorem every automorphism of $U$ is continuous. Applying this to the conjugation action of $G$ on $U$ and using Lemma \ref{going up} we deduce that there is a profinite topology $\tau'$ on $G$ for which $U$ is open. Since $G$ is profinitely rigid, $\tau=\tau'$ and hence $U$ is open. 
	\end{proof}
	
	We extend the result of Kiehlmann to the larger class of profinite groups, which are the profinite groups of finite semisimple length. These are the profinite groups $G$ which admit a finite subnormal series of closed subgroups $1=G_n\unlhd \cdots \unlhd G_1\unlhd G_0=G$ such that all the quotients $G_i/G_{i+1}$ are semisimple. By  successively replacing each $G_i$ with the intersection of its $G$-conjugates we may assume that each $G_i$ is a closed normal subgroup of $G$.

	\begin{thm}\label{ss}
		Every profinite group of finite semisimple length is profinitely rigid.
	\end{thm}
	
	First we need a lemma. For a profinite group $G$ we denote by $G_*$ the largest closed normal semisimple subgroup of $G$. Note that $G_*$ exists since the closure of the product of a family of closed normal semisimple subgroups of G is also normal and semisimple.
	Moreover if $G$ has semisimple length $l$ then $G/G_*$ has semisimple length $l-1$.

	\begin{lem} \label{*} Let $G$ be a profinite group of finite semisimple length. Let $G_*= \prod_{i \in I} S_i$ where $I$ is nonempty and each $S_i$ is a finite nonabelian simple group. Then $G_*= \cap_{i \in I} N_G(S_i)$.
		
	\end{lem}
	
	\begin{proof}
		Let $W= \cap_{i \in I} N_G(S_i)$, this is a closed normal subgroup of $G$ and hence has finite semisimple length. In particular $W$ is topologically (in fact even abstractly) perfect. Clearly $G_* \leq W$ and the action of $W$ on $G_*$ by conjugation induces a continuous homomorphism $f: W \rightarrow \prod_{i \in I} \mathrm{Aut}(S_i)$. Since $\mathrm{Out}(S_i)$ is a finite solvable group and $W$ is topologically perfect $f(W) \leq \prod_{i \in I} \mathrm{Inn}(S_i)$. It follows that
		$W= G_* \times L $ where $L=C_G(G_*)$. We claim  that $L$ is the trivial group. Indeed $L$ is a normal closed subgroup of $G$ hence also has finite semisimple length. If $L \not = \{1\}$ then $L_* \not = \{1\}$ is a characteristic closed normal semisimple subgroup of $L$ and hence $L_* $ is a closed normal subgroup of $G$. It follows that $G_*L_*$ is a normal closed semisimple subgroup of $G$ contradicting the maximality of $G_*$. Hence $L=1$ and $W=G_*$.   
	\end{proof}
	
	We can now prove Theorem \ref{ss}.
	\begin{proof}
		Let $G$ be a profinite group with topology $\tau$. Suppose that $G$ has semisimple length $l$. We will prove that $G$ is profinitely rigid by induction on $l$. The case $l=1$ has been established by J. Kiehlmann. Let $G_*=\prod_{i \in I} S_i$ be the largest normal semisimple closed subgroup of $G$. By induction we may assume that $G/G_*$ is profinitely rigid. Let $\tau'$ be another profinite topology on $G$. Lemma \ref{*} gives $G_*= \cap_{i \in I} N_G(S_i)$. Each $S_i$ is a finite subgroup of $G$ and therefore $S_i$ and $N_G(S_i)$ are $\tau'$-closed subgroups of $G$. It follows that $G_*$ is $\tau'$-closed. 
		
		Let $U$ be a $\tau$-open normal subgroup of $G$. It is sufficient to prove that $U$ is open in $\tau'$. Since $G_*$ is closed in both $\tau$ and $\tau'$ and $G/G_*$ is profinitely rigid we obtain that $G_*U$ is open in $\tau '$. Replacing $G$ with $G_* U$ we may assume that $G_*U=G$. Moreover  $U \cap G_*$ is an open normal subgroup of $G_*$ . Since $G_*$ is profinitely rigid we have that $G_* \cap U$ is open in $G_*$ under the induced topology from $\tau'$ and hence $G_* \cap U$ is $\tau'$-closed subgroup of $G$. By replacing $G$ with $G/(G_* \cap U)$ we may assume that $G_* \cap U=\{1\}$. Together with $G_*U=G$ we deduce that $G=G_* \times U$. Since $G_*$ is a semisimple group it has trivial center and we deduce $U=C_G(G_*)$. Thus $U$ is closed in $\tau'$ and since $|G:U|$ is finite $U$ is open in $\tau'$. Theorem \ref{ss} is proved.
	\end{proof}
	
	Although an extension of a profinite group of finite semisimple length by a profinite group of finite semisimple length remains rigid, this is not true for general profinitely rigid groups as can be shown by the following example:
	\begin{exam}
		Assume there exists a profinitely rigid group $G$  which admits a noncontinuous finite abelian image $A$. Being finite, $A$ is profinitely rigid as well. Look at the direct product $A\times G$. It admits a natural profinite topology when $G$ is considered as a closed subgroup. Now let $\varphi:G\to A$ be a noncontinuous homomorphism. Then $\{\varphi(g)g\}$ is another complement to $A$ inside $A\times G$. We claim that $\{\varphi(g)g\}$ is not closed in $A\times G$. Indeed, since $\varphi:G\to A$ is noncontinuous, there admits an element $h\in G$ such that $h\in \overline{\ker(\varphi)}\setminus \ker(\varphi)$. Assume that $h\in \{\varphi(g)g\}$. There exists an element $x\in G$ such that $h=x\varphi(x)$. So $G\ni x^{-1}h=\varphi(x)\in A$. Thus, $\varphi(x)=1$, meaning that $x\in \ker(\varphi)$ and $h=x$, a contradiction. Hence, $h\notin \{\varphi(g)g\}$. However, as $h\in \overline{\ker(\varphi)}$, for every normal open subgroup $U\unlhd_o G$ there exists $x\in \ker(\varphi)$ such that $hU=xU$. It implies that $hU=x\varphi(x)U$. As the set of all open normal subgroups of $G$ is a basis for the profinite topology on $A\times G$, we get that $h\in \overline{\{\varphi(g)g\}}$. Now, as $\{\varphi(g)g\}\cong G$ it has a (unique) profinite topology. So we can define on $A\times G$ the topology induced from the decomposition $A\times \{\varphi(g)g\}$ and the product topology. This is a different topology, since $\{\varphi(g)g\}$ is closed. 
		
		We are left to prove the existence of such $G$.
		Let $(L_i)_{i \in \mathbb N}$ be a sequence of finite groups such that the word width of $L_i$  with respect to squares tends to infinity. For example we may take $L_i=F_i/[F_i^2,F_i]$, where $F_i$ is the free group of rank $i$. Let $T:=\prod_{i \in \mathbb N} L_i$. The algebraic subgroup $T^2$ generated by all squares of $T$ is not closed in $T$ and therefore $T$ has a nonopen subgroup of index $2$. Hence $T$ has a noncontinuous abelian homomorphic image of size 2. Let $M_i$ be the standard wreath product $A_5 \wr L_i$ and observe that the center of $M_i$ is trivial. Let $G=\prod_{i \in \mathbb N} M_i$. We claim that the profinite topology of $G$ is unique. Let $ j \in \mathbb N$ and let $P_j= \cap_{i \not =j} C_G(M_i)$. Since $M_i$ are finite groups the group $P_j$ is closed in any profinite topology of $G$. Note that $P_j$ is the kernel of the projection $G \rightarrow M_j$. Thus $P_j$ is open with respect to any topology of $G$ and it easily follows that $G$ is  profinitely rigid. Since $G$ maps continuously onto $\prod_{i \in \mathbb N} L_i$ we deduce that $G$ has a noncontinuous image isomorphic to $C_2$.

	\end{exam}
	
	The first examples of profinite groups which are not weakly rigid were given by J. Kiehlmann in \cite{kiehlmann2013classifications}:
	\begin{exam} \label{ee}
		Let $p$ be a prime number. Then the groups $G_1=\prod_{n\in \mathbb{N}}C_{p^n}$ and $G_2=\mathbb Z_p \times \prod_{n\in \mathbb{N}}C_{p^n}$ are abstractly isomorphic, but are not isomorphic as profinite groups. Thus $G_1$ is not weakly rigid.
	\end{exam}
	All of the above examples of nonrigid profinite groups admit an abelian quotient. We use Kiehlmann's example \ref{ee} of groups $G_1$ and $G_2$ in order to construct a perfect profinite group which is not weakly rigid.
	
	Let $W_i$ be the closed subgroup of $G_i^5$ defined as follows.
	
	\[W_i= \{ (x_1, \ldots, x_5) \in G_i^5 \ | \ \sum_{i=1}^5 x_i=1\}.\]
	
	Define an action of the alternating group $A_5$ on $G_i^5$ and on $W_i$ by permutations of the five coordinates. Let $L_i= W_i \rtimes A_5$, this is a profinite group with the topology induced by its open subgroup $W_i$. 
	
	\begin{exam}
		The profinite groups $L_1$ and $L_2$ above are perfect, abstractly isomorphic but not isomorphic as profinite groups.    
	\end{exam}
	
	Proof: To show that $L_i$ is perfect it is enough to show that $a:=(x,x^{-1},1,1,1)$ belongs to $L_i'$ for any $x \in G_i$, since the conjugates of all such $a$ under $A_5$ generate $W_i$.
	
	Let $b=(x,1,x^{-1},1,1) \in W_i$ and $g=(12)(45) \in A_5$, we have that $a=b^{-1}b^g=[b,g] \in L_i'$ as required. 
	
	Let $f:G_1 \rightarrow G_2$ be an abstract isomorphism between $G_1$ and $G_2$. We define an isomorphism $F: L_1 \rightarrow L_2$ by declaring
	$F(g)=g$ for all $g \in A_5$ and $F((x_1, \ldots, x_5))=(f(x_1), \ldots, f(x_5))$ for all $(x_1, \ldots, x_) \in W_1$. It is trivial to check that $F$ is indeed a group isomorphism.
	
	On the other hand suppose that $h: L_1 \rightarrow L_2$ is a continuous isomorphism. Then $h(W_1)=h(W_2)$ since $W_i$ is the largest normal abelian subgroup of $L_i$. It follows that $W_1$ and $W_2$ are isomorphic as profinite groups. This is impossible because $W_1$ is the closure of its torsion subgroup and $W_2$ is not (since $G_2$ and hence $W_2$ have continuous homomorphisms onto $\mathbb Z_p$).

	We say that a finite group is anabelian if it has no abelian composition factors. Similarly a profinite group is said to be anabelian if it is an inverse limit of anabelian finite groups. 
	Note that a profinite anabelian group could have finite index subgroups which are not open, for example see \cite[Proposition 2]{segal2018remarks}. This motivates the following.
	\begin{oques}
		Is every anabelian profinite group profinitely rigid?
	\end{oques}
	
	Let $G$ be a profinite group which is not weakly rigid. This means that $G$ admits two profinite topologies $\tau,\tau'$ such that $(G,\tau)$ and $(G,\tau')$ are not isomorphic as topological groups. How different can these topological group be from each other?  The first easy observation is that they have the same supernatural order. Recall that the supernatural order of a profinite group $G$, $o(G)$, is defined by the lowest common multiple of the orders of all finite continuous quotients of $G$. By \cite[Proposition 4.2.3]{ribes2000profinite}, for every finite abstract quotient $A$ of $G$, $o(A)\mid o(G)$. This implies that $o(G)$ can be described as the lowest common multiple of the orders of all finite quotients of $G$. Thus, $o(G)$ is independent of the profinite topology on $G$. As an immediate result, we get that if for some $p$, a $p$-Sylow subgroup $P_{\tau}$ of $(G,\tau)$ for some profinite topology $\tau$ on $G$ is finite, then for every profinite topology $\tau'$ on $G$, $P_{\tau'}\cong P_{\tau}$. This result in fact holds in general:
	\begin{thm}
		\label{p-sylow invariant}
		Let $\varphi:G\to H$ be an abstract isomorphism of profinite groups. Let $P$ be a $p$-Sylow subgroup of $G$. Then $\varphi(P)$ is a $p$-Sylow subgroup of $H$.
	\end{thm}
	\begin{proof}
		First we show that the image of a $p$-Sylow subgroup of a profinite group is a $p$-Sylow subgroup in every finite abstract quotient.
		
		Let $\varphi:G\to A$ be an abstract epimorphism onto a finite group.  Let $x_1,...,x_n$ be a finite set of preimages of $A$. Look at $H=\overline{\langle x_1,...,x_n\rangle}$. It is a finitely generated profinite group and hence strongly complete. Thus, $\varphi|_H$ is continuous. Now let $P_0\in \operatorname{Syl}_p(H)$. Since $\varphi_H$ is continuous, $\varphi(P_0)\in \operatorname{Syl}_p(A)$. Since $H$ is closed, there exists some $P_1\in \operatorname{Syl}_p(G)$ such that $P_1\cap H=P_0$. So, $\varphi(P_1)\supseteq \varphi(P_0)$. On the other hand, since $P_1$ is a pro-$p$ group, by \cite[Proposition 4.2.3]{ribes2000profinite} $\varphi(P_1)$ is a finite $p$-group. So $\varphi(P_1)=\varphi(P_0)\in \operatorname{Syl}_p(A)$. Finally, $P$ and $P_1$ are conjugate and so are their images, so $\varphi(P)\in \operatorname{Syl}_p(A)$.
		
		Now let $P$ be a $p$-Sylow subgroup of $G$. By the previous claim, the images of $P$ (after applying $\varphi$) in any continuous finite quotient of $H$ are $p$-Sylow subgroups, which are obviously compatible. So $P$ is contained in their inverse limit $Q$ which is a $p$-Sylow subgroup of $H$. By the same argument, $\varphi^{-1}(Q)$ is contained in a $p$-Sylow subgroup $P'$ of $G$. So $P\leq P'$. Since $P$ is a $p$-Sylow subgroup it can not be properly contained in any $p$-subgroup, thus $P=P'$ and we get that $\varphi(P)=Q$.  
	\end{proof}
	As for the rank of the groups we have the following:
	\begin{rem}
		Let $\tau, \tau'$ be two profinite topologies on an abstract group $G$ and assume that $(G,\tau)$ is finitely generated as a topological group, then so is $(G,\tau')$. In fact, these two groups are isomorphic as profinite groups, since $(G,\tau)$ is strongly complete. For topologies of infinite rank, however, having an equality depends on the injectivity of the map $\nn\to 2^{\nn}$ on the class of cardinals, since by \cite{BarOn+2021}, $\omega_0(\hat{G})=2^{2^{\omega_0(G,\tau)}}$. Here $\omega_0(G)$ stands for the \textit{local weight} of $G$, which for profinite groups which are not finitely generated equals to minimal cardinality of a set of generators converging to 1, also referred to as the \textit{rank}.
	\end{rem}
	\section{Connections with cohomological goodness}
	The notions of profinite type and profinite rigidity have a strong connection to group cohomology. 
	
	In \cite{serre1979galois} Serre defined the notion of cohomological goodness as follows. Let $\varphi:G\to K$ be a homomorphism from an abstract group to a profinite group such that $\varphi(G)$ is dense in $K$, and let $M$ be a finite continuous $K$-module. Then $\varphi$ induces a series of maps $\varphi^i:H^i(K,M)\to H^i(G,M)$.  Of particular interest is the case where $K=\hat{G}$ is the profinite completion of $G$. The class $\mathcal{A}_n$ consists of those group $G$ for which $\varphi^i$ is an isomorphism for all $0\leq i\leq n$ and every finite continuous $\hat{G}$-module $M$. An abstract group $G$ is called \textit{cohomologically good} if $G\in \mathcal{A}_n$ for all $n$. Observe that $\mathcal{A}_1$ equals to the class of all abstract groups. A lot of research is devoted to identifying  cohomologically good groups, and in particular the class $\mathcal{A}_2$, as can be seen, for example, at, \cite{lorensen2008groups, grunewald2008cohomological}. In fact we have the following characterization: a residually finite group $G$ belongs to $\mathcal{A}_2$ if and only if every finite extension of $G$ is residually finite (see \cite[Proposition 2.4]{lorensen2008groups}).
	
	An analog of Serre's cohomological goodness which appears in the literature is the following: Let $G$ be a profinite group, $K=G$ and let $\varphi:G\to G$ be the identity map. Let $\mathcal{A}_n^{\text{pro}}$ be the class of all profinite groups $G$ for which $\varphi^n:H^n_{\text{con}}(G,M)\to H^n_{\text{abs}}(G,M)$ is isomorphism for every finite continuous $G$- module.  The maps $\varphi^n$ are also called the \textit{comparison maps} (\cite{fernandez2007comparison}). For pro-$p$ groups, the first comparison map is known to examine whether the group is finitely generated (see \cite{fernandez2007comparison}). Using the proof of Theorem \ref{p-sylow invariant}, we can generalize this result to general pronilpotent groups.
	\begin{prop}
		The pronilpotent groups in  $\mathcal{A}_1^{\text{pro}}$ are precisely the strongly complete pronilpotent groups. 
	\end{prop}
	\begin{proof}
		Let $G$ be a pronilpotent group.
		
		First assume that $G$ is strongly complete. Then $\hat{G}\cong G$ where the natural homomorphism is just the identity map. Then since every group belongs to $\mathcal{A}_1$ we are done.
		
		Now assume that $G\in \mathcal{A}_1^{\text{pro}}$. Let $P$ be the unique $p$-Sylow subgroup of $G$, and choose $M=\F_p$ to be a module with the trivial action. Since $\varphi^1:H^1_{\text{con}}(\hat{G},\F_p)\to H^1_{\text{abs}}(G,\F_p)$ is always an isomorphism, it is equivalent to saying that $\varphi^1:H^1_{\text{con}}(\hat{G},\F_p)\to H^1_{\text{con}}(G,\F_p)$ is an isomorphism. Recall that by 
		the decomposition of $G$ to direct product $G\cong P\times P'$, for a $p$-Sylow subgroup $P$ and a complement $P'$, the inclusion map induces an isomorphism $H^1_{\text{con}}(G,\F_p)\to H^1_{\text{con}}(P,\F_p)$. 
		Since $G\cong P\times P'$, where $P'$ is of order prime to $p$, then $\hat{G}\cong \hat{P}\times \hat{P'}$ and by \cite[Proposition 4.2.3]{ribes2000profinite} $\hat{P}$ is $p$-Sylow subgroup of $\hat{G}$. Again the natural map $H^1_{\text{con}}(\hat{G},\F_p)\to H^1_{\text{con}}(\hat{P},\F_p)$ is an isomorphism, and hence so is the natural map $H^1_{\text{con}}(P,\F_p)\to H^1_{\text{abs}}(P,\F_p)\cong H^1_{\text{con}}(\hat{P},\F_p)$. Hence by \cite{fernandez2007comparison} $P$ is strongly complete. This holds for every $p$-Sylow subgroup of $G$. We are left to show that if $G$ is a pronilpotent  group such that for every $p$, the $p$-Sylow subgroup of $G$ is strongly complete, then so is $G$. This will follow from the next general result.
	\end{proof}

	\begin{prop}
		Let $G$ be a profinite group such that for every $p$, the $p$-Sylow subgroups of $G$ are strongly complete, then $G$ is strongly complete.
	\end{prop}
	\begin{proof}
		First step: Let $G$ be a profinite group and $P$ a $p$-Sylow subgroup of $G$, then $\Bar{P}$, the closure of $P$ in $\hat{G}$, is a $p$- Sylow subgroup of $\hat{G}$. Indeed, that follows since $\bar{P}\cong \invlim \psi_i(P)$, for $\psi_i:G\to A_i$ runs over all the finite abstract quotients of $G$, is an inverse limit of $p$-Sylow subgroups, by the proof of Proposition \ref{p-sylow invariant}. 
		
		Second step: Assume that $\varphi|_P:P\to \bar{P}$ is an isomorphism. The identity map $\operatorname{id}:G\to G$ together with the universal property of the profinite completion (see \cite[Chapter 3]{ribes2000profinite}), yields the following commutative diagram:
		\[
		\xymatrix@R=14pt{ G\ \ar[rd]^{\hat{id}}& \\
			G \ar[u]_{\varphi} \ \ar[r]_{id}& G\\
		}
		\]
		such that for every $g\in G$, $\hat{id}(\varphi(g))=g$. Let $K=\ker(\hat{id})$ and $Q=\bar{P}\cap K$. Then $Q$ is a $p$-Sylow subgroup of $K$. Since $\varphi:P\to \bar{P}$ is an isomorphism, we conclude that $Q=0$. Applying this result for every prime $p$, we get that $K=0$, i.e, $\varphi:G\to \hat{G}$ is an isomorphism.
	\end{proof}
	The following Lemma is an immediate application of the correspondence between the second cohomology group and group extensions:
	\begin{lem}
		Let $(G,\tau)$ be a profinite group and $M$ a finite continuous module. Denote by $\varphi^2:  H^2_{\text{con}}(G,M)\to H^2_{\text{abs}}(G,M)$ the natural homomorphism. Then
		\begin{enumerate}
			\item $\varphi^2$ is surjective if and only if for every group extension $1\to M\to H\to G\to 1$, $H$ admits a profinite topology which induces $\tau$ on $G$.
			\item $\varphi^2$ is injective if and only if for every profinite extension $H$ of $G$ by $M$, $H$ admits a unique profinite type which induces $\tau$ on $G$ via the quotient topology. I.e, for every pair of profinite topologies $T_1,T_2$ on $H$ such that the quotient map $(H,T_i)\to (G,\tau)$ is continuous, there exists a continuous isomorphism $(H,T_1)\to (H,T_2)$ which is compatible with the quotient map.
		\end{enumerate}
	\end{lem}
	\begin{proof}
		\begin{enumerate}
			\item Let $c\in H^2_{\text{abs}}(G,M)$. Choose some preimage $\tilde{c}$ of $c$ in $C^2_{\text{abs}}(G,M)$. $\tilde{c}$ corresponds to some abstract group extension $1\to M\to H\to G\to 1$. Assume $H$ is given a profinite topology which is compatible with $\tau$. Hence, $H$ corresponds to some element $c'\in H^2(G,M)$. Let $\tilde{c'}$ be a preimage of $c'$ in $C^2(G,M)$. Since $\tilde{c}$ and $\tilde{c'}$ correspond to the same group extension, they are equivalent in $ H^2_{\text{abs}}(G,M)$, i.e, $i_2(c')=c$.
			
			For the second direction, assume that $i_2$ is onto. Let $1\to M\to H\to G\to 1$ be an abstract group extension. It corresponds to some element $c\in C^2_{\operatorname{abs}}(G,M)$. By assumption, $c$ is equivalent in $H^2_{\operatorname{abs}}(G,M)$ to some element $c'\in H^2(G,M)$. But $c'$ corresponds to a profinite group extension $1\to M\to H'\to G\to 1$. Hence, these extensions are equivalent. In particular, $H\cong H'$, which means that $H$ can be given a profinite topology compatible with $\tau$. 
			\item Let $H_1,H_2$ be profinite extensions of $G$ by $M$ such that $H_1\cong_{\operatorname{abs}} H_2$, as extensions. Each one of them corresponds to an element $c_1,c_2\in H^2_{\text{con}}(G,M)$, correspondingly. Notice that the images of $c_1,c_2$ in $H^2_{\operatorname{abs}} (G,M)$ become equal by the abstract isomorphism. By injectivity of $\varphi^2$ we get that the coset of $c_1,c_2$ in $H^2(G,M)$ are equal, which means that $H_1\cong H_2$ continuously. The second direction is similar.
		\end{enumerate}
	\end{proof}
	A lot of effort is devoted to classifying those pro-$p$ groups for which the second comparison map is an isomorphism. In \cite{sury1993central} Surry proved the isomorphism holds for every solvable and Chevalley $p$-adic analytic group, and conjectured it holds for every $p$-adic analytic group. For profinitely rigid groups these properties can be stated as follows:
	\begin{cor}
		Let $G$ be a rigid profinite group. The map $\varphi^2$ is surjective if and only if every finite extension of $G$ is of profinite type, while $\varphi^2$ is injective if and only if every finite continuous extension of $G$ is weakly rigid.
	\end{cor}
	\begin{exam}
		Let $G$ be a strongly complete profinite group, and $1\to M\to H\to G$ a finite extension of $G$. Assume that $H$ is residually finite. Then taking profinite completion we get an exact sequence $1\to M\to \hat{H}\to G\to 1$, which implies the natural homomorphism $H\to \hat{H}$ being isomorphism. Hence $H$ is of profinite type, and moreover $H$ is strongly complete and thus rigid. We conclude that $G\in \mathcal{A}_2^{pro}$ if and only if $G\in \mathcal{A}_2$. 
	\end{exam}
	We end the paper with the following result.

	For the class $\mathcal{A}_2$, Serre proved the following equivalence: $G$ belongs to $\mathcal{A}_2$ if and only if for every extension $1\to N\to E\to G\to 1$ where $N$ is finitely generated, the induced map $\hat{N}\to \hat{E}$ is injective. The properties of $\varphi^2$ in the contexts $\mathcal{A}_2^{\text{pro}}$ can also be stated in terms of extensions with profinite finitely generated kernel:
	\begin{lem}\label{from finite to f.g}
		Let $(G,\tau)$ be a profinite group such that for every group extension with finite kernel $1\to A\to H\to G\to 1$, $H$ admits a unique profinite topology inducing $\tau$ on $G$. Then for every group extension $1\to A\to H\to G\to 1$ with $(A,\tau')$ finitely generated, $H$ admits a unique profinite topology inducing $\tau$ on $G$ and $\tau'$ on $A$.
	\end{lem}
	\begin{proof}
		Let $1\to A\to H\to G\to 1$ be a group extension such that $(A,\tau'),(G,\tau)$ are profinite and $(A,\tau')$ is finitely generated. Then $A$ admits a series of open characteristic subgroups $A=A_0\geq A_1\geq A_2\geq \cdots \geq A_n\geq \cdots $ such that $\bigcap_{n\in \mathbb{N}} A_n={e}$ (see, for example \cite[Proposition 2.5.1]{ribes2000profinite}). In particular, for every $n$, $A_n\unlhd H$. Let $n$ be a natural number, then $1\to A/A_n\to H/A_n\to G\to 1$ is a group extension with $A/A_n$ finite. So $H/A_n$ admits a profinite topology. The natural projections $H/A_n\to H/A_m$ are continuous since $H/A_n$ induces on $H/A_m$ a topology compatible with $\tau$, and $H/A_m$ admits a unique such topology, by assumption. Since $A\cong \invlim A/A_n$, $H\cong H/A_n$. Indeed, let $(h_n)$ be a compatible series in $\prod H/A_n$, then every component has the form $c_na_n$ where $c_n$ is some representative of $A$ in $H$ and $a_n$ is some representative of $A_n$ in $A$. Since the series is compatible, $c_n=c$ is fixed for all $n$. The series of the $a_n$ is compatible and thus yields to an elements in $A$. So we can define a profinite topology on $H$, compatible with $\tau$. On the other hand, for every profinite topology on $H$ for which $A$ is closed, the groups $A_n$ must be closed too as an open subgroup of $A$, and so $H\cong {\invlim}_n H/A_n$. Since the topology on $H$ induces $\tau$ on $G$, so thus the quotient topologies on $H/A_n$, so the quotients topologies on $H/A_n$ are uniquely determined, and hence so is the topology on their inverse limit.
	\end{proof}

	\bibliographystyle{plain}

\end{document}